\documentclass[11pt]{article}
\usepackage[T1]{fontenc}
\usepackage{amsmath,amsthm,mathtools}
\usepackage{libertine}
\usepackage[libertine]{newtxmath}
\usepackage{microtype,booktabs}
\usepackage[letterpaper,margin=1in]{geometry}
\usepackage[colorlinks=true,linkcolor=blue,citecolor=blue,urlcolor=blue]{hyperref}
\usepackage[capitalise,nameinlink,noabbrev]{cleveref}
\Crefformat{equation}{#2(#1)#3}
\Crefrangeformat{equation}{#3(#1)#4--#5(#2)#6}
\Crefmultiformat{equation}{#2(#1)#3}{ and #2(#1)#3}{, #2(#1)#3}{, and #2(#1)#3}
\hypersetup{pdftitle={Retained-Set Descent for Diagonal Ramsey Numbers},pdfauthor={Zhipeng Lu and Sichen Wang},pdfsubject={Finite Ramsey proofs, their greatest fixed point, and rigorous upper and lower bounds}}
\newtheorem{theorem}{Theorem}
\newtheorem{lemma}{Lemma}[section]
\newtheorem{proposition}[lemma]{Proposition}
\newtheorem{corollary}[lemma]{Corollary}
\theoremstyle{definition}
\newtheorem{definition}[lemma]{Definition}
\newcommand{\e}{\mathrm e}
\newcommand{\BB}{\mathcal B}
\newcommand{\DD}{\mathsf D}
\newcommand{\pos}[1]{(#1)_+}
\newcommand{\zp}{\frac{1306998938417}{10^{12}}}
\newcommand{\zm}{\frac{1305460275532}{10^{12}}}
\title{Retained-Set Descent for Diagonal Ramsey Numbers}
\author{Zhipeng Lu\\
  \small Shenzhen MSU-BIT University\\
  \small\texttt{zhipeng.lu@smbu.edu.cn}
  \and
  Sichen Wang\\
  \small Shenzhen MSU-BIT University\\
  \small\texttt{wsc@smbu.edu.cn}}
\date{}
\begin{document}
\maketitle
\begin{abstract}
We study how far a fixed Ramsey upper bound can be improved by descending through blue neighborhoods in one vertex set while keeping a second set fixed.
A weighted inequality in the two set sizes determines when the descent can stop.
For the source bound specified here, the infimum diagonal exponent over all finite derivations lies in $[1.305,\,1.307]$.
A finite derivation gives $R(k,k)\le3.69507^k$ for all sufficiently large $k$; a concave polygon proves the lower bound for every finite depth.
We also characterize the infimum as a greatest fixed point and show that every larger exponent has a finite derivation valid uniformly for nearby clique-size ratios.
\end{abstract}

\section{Introduction}
Let $R(k,\ell)$ be the least integer $N$ such that every red--blue coloring of $K_N$ contains a red $K_k$ or a blue $K_\ell$.
Ramsey's theorem~\cite{Ramsey} guarantees that these numbers are finite.
The Erd\H{o}s--Szekeres recurrence~\cite{ES} gives $R(k,\ell)\le\binom{k+\ell-2}{k-1}$, and hence $R(k,k)\le4^k$.
In the other direction, Erd\H{o}s's random-coloring argument~\cite{Erdos} gives a lower bound of order $k2^{k/2}$; Spencer~\cite{Spencer} improved its leading constant using the Lov\'asz local lemma.

Thomason~\cite{Thomason} improved the classical upper bound by a polynomial factor using quasirandomness: a coloring close to the inductive bound must have degrees and small-subgraph counts close to those of a random coloring.
Conlon~\cite{Conlon} developed this approach to obtain $R(k,k)\le4^k\exp(-c(\log k)^2/\log\log k)$.
Sah~\cite{Sah} strengthened the estimate to $R(k,k)\le4^k\exp(-c(\log k)^2)$.
Each bound holds for a suitable absolute constant $c>0$ and all sufficiently large $k$.
These results gave successively larger subexponential improvements while retaining exponential base $4$.

Campos, Griffiths, Morris, and Sahasrabudhe~\cite{CGMS} obtained the first improvement in the exponential base, proving $R(k,k)\le(4-\varepsilon)^k$ for an absolute $\varepsilon>0$ and all sufficiently large $k$.
Their book algorithm works with two vertex sets and controls the red density between them.
Gupta, Ndiaye, Norin, and Wei~\cite{GNNW} replaced this algorithm by a candidate induction, obtained stronger off-diagonal estimates, and optimized the resulting parameters to give $R(k,k)\le3.8^{k+o(k)}$.
Their candidate induction and reuse of improved Ramsey bounds in this optimization are the starting points of our work.

We improve the diagonal bound to $3.69507^k$ for all sufficiently large $k$.
Our descent applies density bounds inside a vertex set $X$ while keeping a second set $Y$ fixed for a final inequality in $|X|^w|Y|$, with $w>0$.
Each resulting bound can be used inside $X$ in a later descent.
We also ask how small an exponent finitely many such repetitions can give from a fixed initial bound.

We call this initial bound the \emph{source}.
It has the form $\log R(a,b)\le a f(b/a)+o(a)$ for integers $1\le b\le a$, where $f$ is a fixed concave function on $[0,1]$ and the error is uniform in $b$.
\Cref{sec:inputs} specifies $f$, and \Cref{sec:source-proof} proves the bound.

A \emph{conditional bound} at density $p\in(0,1)$ and ratio $s\in(0,1]$ applies to colorings with at least a fraction $p$ of their edges red.
An exponent $z$ means that, for some fixed $C\ge1$ and all sufficiently large $k$, order at least $C\e^{zk}$ guarantees a red $K_k$ or a blue $K_{\lfloor sk\rfloor}$.
We use families of such bounds on closed intervals of $s$, with continuous exponents and common constants and thresholds.

Starting from the conditional bounds supplied by the source, we apply the descent rule using bounds proved at earlier steps.
We may also restrict ratio intervals, increase a density threshold or an exponent, and combine bounds at the same density on a finite closed cover of a ratio interval.
A \emph{finite proof} uses only finitely many of these operations; \Cref{def:finite-proof} gives the precise rules.

Let $U_{\rm fin}(p,s)$ be the infimum of the exponents obtained at density $p$ and ratio $s$ by such finite proofs.
In the diagonal case $(p,s)=(1/2,1)$, we construct a finite proof with rational exponent $z_+$ and a rational lower bound $z_-$ for all finite proofs.
Their exact values are given in \Cref{sec:certificate}.

\begin{theorem}\label{thm:main}
For this source and these finite proof rules,
\begin{equation}\label{eq:bracket}
 1.305<z_-\le U_{\rm fin}(1/2,1)\le z_+<1.307.
\end{equation}
The exponent $z_+$ has a finite proof.
For all sufficiently large integers $k$,
\begin{equation}\label{eq:main}
 R(k,k)\le\left\lceil150003\exp(z_+k)\right\rceil,
\end{equation}
and hence $R(k,k)\le3.69507^k$.
\end{theorem}

For the Ramsey bound, choose the denser color as red and apply the conditional bound at $(p,s)=(1/2,1)$.

\subsection{Proof Outline}
We use disjoint sets $X,Y$, keeping $Y$ fixed and requiring each vertex of $X$ to have at least a fixed positive fraction of $Y$ as red neighbors.
The red target $k$ and the blue target $\ell$ in $Y$ stay fixed while the blue target $t$ in $X$ decreases.
When $X$ is large enough for a conditional bound at density $q$, we either finish at red density at least $q$ or descend to a blue neighborhood of size greater than $(1-q)(|X|-1)$.
Restricting $X$ to this neighborhood preserves its red degrees into $Y$ and reduces $t$ by one: a blue $K_{t-1}$ there extends through the chosen vertex.

Write the required size of $X$ at $s=t/k$ as $\exp(kg(s))$, with $g$ piecewise affine.
On a cell of slope $d$, the threshold falls by a factor $\e^{-d}$ per step, so we take $d>-\log(1-q)$ to allow for the neighborhood loss.
The profile $g$ must also exceed the exponents of the bounds used inside $X$.
With slopes and calls fixed, optimizing its vertical shift under these conditions and the initial and terminal inequalities gives the cost in~\Cref{eq:route-optimum}.

Reusing these bounds and iterating from the source gives a decreasing sequence of exponents whose limit is a greatest fixed point (\Cref{thm:closure}).
Strict inequalities persist on neighborhoods of ratios; compactness selects finitely many neighborhoods and hence a common finite depth for a whole descent.
Hence every exponent strictly above the limit has a finite proof, uniformly near its target ratio.

For the lower bound, a concave polygon $F$ assigns exponent $F(s)$ at densities up to $1-\e^{-d}$ on a segment of slope $d$.
Where a descent using bounds that respect this estimate has $g<F$, its call density must exceed $1-\e^{-d}$, forcing the slope of $g$ above $d$.
Thus $g-F$ decreases as the descent moves to smaller ratios.
We choose $F$ so that the terminal weighted inequality cannot hold below it.
The supporting-line formula reduces this requirement to one-variable inequalities at its vertices (\Cref{thm:polygon}); the certificates in \Cref{sec:certificate} complete the proof of \Cref{thm:main}.

The Lean 4 formalization and the Python code for exact certificate checks are available in the accompanying repository.\footnote{\url{https://github.com/sichen-wang/diagonal-ramsey-numbers}}

\section{Source Bounds}\label{sec:inputs}
We now describe the source function and express its bound through pairs $(x,y)$.
These pairs supply the weighted inequality and the initial conditional bounds used in the descent.

All logarithms are natural. Let
\[
 h(s)=(1+s)\log(1+s)-s\log s,\qquad h(0)=0.
\]
The classical Ramsey bound gives $\log R(k,\ell)\le kh(\ell/k)$ for $1\le\ell\le k$~\cite{ES}.

Our fixed source is
\begin{equation}\label{eq:source-f}
 f(s)=h(s)+s\e^{-s}P(s)\qquad(0\le s\le1),
\end{equation}
where the piecewise cubic $P$ is specified by the rational partition points, values, and derivatives in \path{data/source.json}.
This source satisfies
\begin{equation}\label{eq:source-rate}
 \log R(a,b)\le a f(b/a)+o(a)\qquad(1\le b\le a),
\end{equation}
with an error uniform in $b$.
\Cref{sec:source-proof} proves this bound and the properties of $f$ used below.

\subsection{Source Pairs and Supporting Lines}
The weighted inequality uses bounds of the form $R(a,b)\le x^{-a}y^{-b}$.
We say that $(x,y)\in(0,1)^2$ is an \emph{admissible Ramsey pair} if this bound holds for all sufficiently large $a+b$.
Color symmetry in~\Cref{eq:source-rate} leads to the extension
\[
 \widehat f(t)=
 \begin{cases}f(t),&0\le t\le1,\\ t f(1/t),&t\ge1.\end{cases}
\]
The source exponent for arbitrary positive targets $a,b$ is then $b\widehat f(a/b)$.
We call $(x,y)\in(0,1)^2$ a \emph{source pair} if
$b\widehat f(a/b)\le-a\log x-b\log y$ for all $a,b>0$, and write $\BB_f$ for their set.
Every interior point of $\BB_f$ is admissible: choose a slightly larger pair still in $\BB_f$ and apply~\Cref{eq:source-rate} with color symmetry.
The resulting error is $o(a+b)$, which the strict coordinate gaps absorb.

Supporting lines give a convenient description of these pairs.
The source verification shows that $f$ is continuous on $[0,1]$ and is $C^1$ on $(0,1]$, where it is strictly concave and increasing.
It also gives $f(0)=0$ and $2f'(1)>f(1)$.
It follows that $\widehat f$ is increasing and concave.
On each smooth piece with $t>1$, $\widehat f''(t)=t^{-3}f''(1/t)<0$; its first derivative is continuous at the reflected source partition points. At $t=1$, the derivative drops from $f'(1)$ to $f(1)-f'(1)$.
Strict concavity and $f(0)=0$ give $f(s)-sf'(s)>0$, so the reflected branch is increasing.
The expression in~\Cref{eq:source-f} gives $\widehat f'(0+)=\infty$, $\widehat f'(\infty)=0$, and $\widehat f(t)/t\to0$ as $t\to\infty$.

For $A>0$, define the intercept of the supporting line with slope $A$ by
\begin{equation}\label{eq:dual}
 b(A)=\sup_{t\ge0}\{\widehat f(t)-At\},\qquad
 \widehat f(t)=\inf_{A>0}\{At+b(A)\}\quad(t>0).
\end{equation}
For fixed $A>0$, $\widehat f(t)-At\to-\infty$ as $t\to\infty$, while $f(t)/t\to\infty$ as $t\downarrow0$ makes this expression positive for small positive $t$.
By continuity, the supremum is positive and attained at some $t>0$.
At each $t>0$, concavity supplies a supporting slope $A>0$: use the derivative, or any slope between the one-sided derivatives at $t=1$.
For this slope $b(A)=\widehat f(t)-At$, proving the second identity.

Writing $x=\e^{-A}$ and $y=\e^{-B}$, the source-pair condition becomes $B\ge b(A)$, so
\[
 \BB_f=\{(\e^{-A},\e^{-B}): A>0,\ B\ge b(A)\}.
\]
The finite convex function $b$ is continuous on $A>0$, so interior source pairs correspond to $B>b(A)$.
Taking $(a,b)=(1,s)$ in the support inequality gives
\begin{equation}\label{eq:support-lower}
 A+s b(A)\ge f(s)\qquad(0<s\le1).
\end{equation}

\subsection{The Weighted Candidate Bound}\label{sec:candidate}
Write $N_\chi(v)$ for the neighborhood of $v$ in color $\chi\in\{R,B\}$ and $\deg_\chi(v,Z)=|N_\chi(v)\cap Z|$.
For disjoint nonempty vertex sets $X,Y$, let $e_R(X,Y)$ be the number of red edges between them and write $d_R(X,Y)=e_R(X,Y)/(|X||Y|)$.
Following~\cite{GNNW}, a \emph{candidate} is an ordered pair $(X,Y)$ of such sets.
We say that a candidate is \emph{$(k,\ell,t)$-good} if $X\cup Y$ contains a red $K_k$, or $X$ contains a blue $K_t$, or $Y$ contains a blue $K_\ell$.
The target in the retained set $Y$ is $\ell$, while the target $t$ in $X$ may decrease.

\begin{lemma}[Weighted Candidate Bound]\label{prop:candidate}
Fix $w>0$ and $\mu,p,x,y\in(0,1)$ with $(x,y)\in\operatorname{int}\BB_f$ and
\[
 x<(1-\mu)^w p^{1/(1-\mu)}.
\]
For all sufficiently large $\ell$, uniformly in positive integers $k,t$,
\[
 d_R(X,Y)\ge p,\qquad
 |X|^w|Y|\ge x^{-k}y^{-\ell}\mu^{-wt}
\]
imply that $(X,Y)$ is $(k,\ell,t)$-good.
\end{lemma}
\begin{proof}
Choose a slightly larger pair $(\widehat x,\widehat y)$ still in $\operatorname{int}\BB_f$.
Once $\ell$ is large, admissibility guarantees a red $K_a$ or a blue $K_\ell$ in every vertex set of size at least $\widehat x^{-a}\widehat y^{-\ell}$, uniformly for $a\ge1$.
Applying \Cref{lem:local} with $C=1$ proves the claim.
\end{proof}

\subsection{Conditional Bounds}
To reuse a bound during a descent, we need it for an interval of blue-to-red target ratios, with a single size constant and a single threshold for $k$.
\begin{definition}[Conditional Bound]\label{def:conditional}
For $p\in(0,1)$, a compact interval $I\subset(0,1]$, and a continuous function $L:I\to(0,\infty)$, we say that the \emph{conditional bound} $\DD(p,L,I)$ holds if some $C\ge1,K$ have the following property uniformly for $r\in I$ and integers $k\ge K$:
a coloring of red density at least $p$ and order at least $C\e^{kL(r)}$ contains a red $K_k$ or a blue $K_{\lfloor rk\rfloor}$.
We choose $K$ so that all blue targets are positive.
For a constant exponent $z$, the notation $\DD(p,z,I)$ means $L\equiv z$.
\end{definition}

We call $(w,\mu,x,y)$ a \emph{strict control at density $p$} when it satisfies the parameter conditions of \Cref{prop:candidate}.
Write $A=-\log x$, $B=-\log y$, and $\beta=-\log\mu$.
Apply the weighted candidate bound with $t=\ell$ to a balanced partition whose red cross-density is at least $p$; such a partition exists by averaging.
Both set sizes are a constant fraction of the graph's order, so the corresponding exponent is
\begin{equation}\label{eq:leaf}
 L(r)=\frac{A+rB+wr\beta}{1+w}.
\end{equation}
Choose $y_1>y$ with $(x,y_1)\in\operatorname{int}\BB_f$.
For $\ell=\lfloor rk\rfloor$ and a host of order $N\ge\e^{kL(r)}$, the partition satisfies
\[
 |X|^w|Y|\ge\frac{x^{-k}y^{-\ell}\mu^{-w\ell}}{3^{w+1}}
 \ge x^{-k}y_1^{-\ell}\mu^{-w\ell}.
\]
The first inequality uses $kr\ge\ell$ and $B+w\beta>0$; the second holds once $(y_1/y)^\ell\ge3^{w+1}$.
Applying \Cref{prop:candidate} at $(x,y_1)$ proves $\DD(p,L,I)$ with $C=1$ on every compact $I\subset(0,1]$.
The threshold is uniform in $r\in I$, since $\ell\ge k\min I-1$.
These \emph{weighted source bounds} are the starting points of our finite derivations, and we call them \emph{source leaves}.

\section{Retained-Set Descent and Its Cost}\label{sec:routes}
We now allow the conditional bound used in $X$ to change as its blue target decreases.
Throughout the descent we keep $\deg_R(v,Y)\ge\pi|Y|$ for every $v\in X$, for a fixed $\pi\in(0,1)$.
This gives cross-density at least $\pi$ after every restriction of $X$.
A positive function $g(s)$ specifies the required size $\exp(kg(s))$ of $X$ at target $t=sk$.
To use a bound $\DD(q,L,I)$, we require $g>L$ to absorb its size constant; the slope condition below absorbs the loss in a blue-neighborhood step.
\begin{definition}[Route]\label{def:route}
A \emph{route} consists of a partition
$0<\theta=u_0<\cdots<u_m=\lambda\le1$ and, on each closed cell $J_j=[u_{j-1},u_j]$, a proved call $\DD(q_j,L_j,I_j)$ with $J_j\subseteq I_j$.
It also specifies slopes
\[
 d_j>c(q_j):=-\log(1-q_j)
\]
and a continuous positive function $g$, constant on $[0,\theta]$ and affine with slope $d_j$ on $J_j$, such that $g>L_j$ on every closed cell.
\end{definition}
The descent starts at ratio $\lambda$ and stops at $\theta$.
Write $d(t)$ for the piecewise constant slope on $[\theta,\lambda]$, taking the right-hand slope at $\theta$ and internal partition points and the left-hand slope at $\lambda$.

\begin{theorem}[Retained-Set Rule]\label{thm:route}
Fix a route as above, an outer density $p$, and a row density $0<\pi<p<1$.
Choose a strict control at density $\pi$ as the route's \emph{terminal parameters}.
If
\begin{equation}\label{eq:route-test}
 z>g(\lambda),\qquad
 z+w g(\theta)>A+\lambda B+w\theta\beta,
\end{equation}
then $\DD(p,z,\{\lambda\})$ holds with
\begin{equation}\label{eq:C}
 C=\frac{3(1-\pi)}{p-\pi}.
\end{equation}
The conclusion is uniform over truncations and vertical shifts of one fixed route when the new endpoint, output exponent, and shift depend continuously on a parameter in a compact set; the stopping ratio, slopes, densities, calls, and terminal controls stay fixed.
At endpoint $r$, retain the closed cells $[u_{j-1},\min\{u_j,r\}]$ with $u_{j-1}<r$ and require positive uniform margins in $g>L_j$ and in~\Cref{eq:route-test}.
Such a family may include $r=\theta$, with no internal calls.
\end{theorem}

\begin{proof}
If $\lambda=\theta$, the size tests give
\[
 (1+w)z>z+wg(\theta)>A+\theta B+w\theta\beta.
\]
The source leaf at density $\pi$ proves the conclusion at density $p>\pi$, with constant $1\le C$.

Suppose $\theta<\lambda$, and take a graph of red density at least $p$ and order $N\ge C\e^{kz}$.
Averaging gives a balanced partition $X_0,Y$ with cross-density at least $p$ and both parts of size at least $N/3$.
Keep in $X_0$ only the vertices with at least $\pi|Y|$ red neighbors in $Y$.
The surviving set $X$ satisfies
\[
 (p-\pi)|X_0||Y|
 \le e_R(X,Y)-\pi|X||Y|
 \le(1-\pi)|X||Y|.
\]
The choice of $C$ gives $|X|,|Y|\ge\e^{kz}$.
Every nonempty subset of $X$ has red cross-density at least $\pi$ with this fixed $Y$.

Put $n_0=\lfloor\theta k\rfloor$ and $\ell=\lfloor\lambda k\rfloor$, taking $k$ large enough that $n_0\ge1$.
To track the loss, define for integers $0\le t\le\ell$
\[
 \Gamma_k(t)=kg(\theta)+\sum_{j=n_0+1}^{t}d(j/k)\quad(t>n_0),
 \qquad \Gamma_k(t)=kg(\theta)\quad(0\le t\le n_0).
\]
For $V=d_1+\sum_{j<m}|d_{j+1}-d_j|$,
\begin{equation}\label{eq:discrete}
 \Gamma_k(t)-\Gamma_k(t-1)=d(t/k)\quad(t>n_0),\qquad
 |\Gamma_k(t)-kg(t/k)|\le V.
\end{equation}
For the second bound, express $d$ as its initial step at $\theta$ and its subsequent jumps.
For each unit step, the mesh count differs from $k$ times its integral by at most one, even when a jump lies on the mesh.
Summing the absolute jump sizes gives $V$.

Let $C_j$ be the call constants and put
\[
 \delta=\min_j\min_{J_j}(g-L_j)>0,\qquad
 \eta=\min_j(1-q_j-\e^{-d_j})>0.
\]
Choose $k$ so that $k\delta>V+\max_j\log C_j$ and $\eta\e^{kg(\theta)}>1$, and so that all call and candidate thresholds hold.
We prove by induction on $t=1,\ldots,\ell$ that $(Z,Y)$ is $(k,\ell,t)$-good whenever $Z\subseteq X$ and $|Z|\ge\e^{\Gamma_k(t)}$.

For $t\le n_0$, the terminal test gives
\[
 \log(|Z|^w|Y|)\ge k[wg(\theta)+z]
 >k(A+\lambda B+w\theta\beta)
 \ge kA+\ell B+wt\beta.
\]
Since $d_R(Z,Y)\ge\pi$, \Cref{prop:candidate} applies.

For $t>n_0$, put $s=t/k$ and use the cell with assigned slope $d_j$.
If the red density in $Z$ is at least $q_j$, then
$\log|Z|\ge kg(s)-V>kL_j(s)+\log C_j$.
The call gives a red $K_k$ or a blue $K_t$.
Otherwise some vertex has a blue neighborhood $Z'\subseteq Z$ with
\[
 |Z'|>(\e^{-d_j}+\eta)(|Z|-1)>\e^{-d_j}|Z|
 \ge\e^{\Gamma_k(t-1)}.
\]
The middle inequality follows from $\eta|Z|>1$ and $\e^{-d_j}+\eta<1$.
Induction applies to $(Z',Y)$.
A blue $K_{t-1}$ extends through the chosen vertex; either of the other two outcomes already proves the claim.

Finally, \Cref{eq:discrete} and the endpoint test give
$\Gamma_k(\ell)\le kg(\lambda)+V<kz$ for large $k$.
The claim therefore applies to $Z=X$ at $t=\ell$.

For uniformity, take the minimum stopping height and the minima of the strict gaps over the compact parameter set, and the maxima of the finitely many call thresholds.
The same $V$ bounds every truncated profile, since truncation only removes slope jumps.
These choices give a common threshold for $k$; at $\lambda=\theta$ use the fixed source control as above.
\end{proof}

\subsection{The Cost of a Descent}\label{sec:route-cost}
Fix the partition, calls, slopes, and terminal controls, and vary the vertical position of $g$.
Put
\begin{equation}\label{eq:route-costs}
 S(s)=\int_s^\lambda d(t)\,dt,\qquad
 M=\max_j\max_{s\in J_j}\{L_j(s)+S(s)\}.
\end{equation}
The quantity $S(s)$ is the cumulative logarithmic loss in descending from $\lambda$ to $s$.
Every $g$ with the specified slopes has the form $g(s)=H-S(s)$ on $[\theta,\lambda]$.
The internal conditions are exactly $H>M$; they also make $g$ positive, since $M\ge L_1(\theta)+S(\theta)>S(\theta)$.
The conditions in~\Cref{eq:route-test} now read
\[
 H<z,\qquad wH+z>A+\lambda B+w\theta\beta+wS(\theta).
\]
There is such an $H>M$ exactly when
\begin{equation}\label{eq:route-optimum}
 z>Z:=\max\left\{M,
  \frac{A+\lambda B+w\theta\beta+wS(\theta)}{1+w}\right\}.
\end{equation}
Necessity follows from $z>H>M$ and $(1+w)z>z+wH$.
Conversely, if $z>Z$, choose $H<z$ sufficiently close to $z$ to satisfy both inequalities.
Thus $Z$ is the infimum output exponent over these vertical positions.
The term $M$ accounts for every internal call together with the loss needed to reach it; the other term is the least exponent compatible with the weighted terminal inequality.

For one call $\DD(q,L,[\theta,\lambda])$ with affine $L$, choose a slope $d>c(q)$.
Here $S(s)=d(\lambda-s)$, so $L+S$ is affine and its maximum occurs at an endpoint.
Substitution in~\Cref{eq:route-optimum} gives
\[
 Z=\max\left\{L(\lambda),\ L(\theta)+d(\lambda-\theta),\
 \frac{A+\lambda B+w\theta\beta+wd(\lambda-\theta)}{1+w}\right\}.
\]
The first two terms require enough vertices to use the call at both ends of the interval; the second includes all the loss incurred before reaching $\theta$.
The third requires enough vertices to apply the weighted bound at $\theta$.
Every $z>Z$ therefore gives a new conditional bound at $\lambda$.

\subsection{Reuse on an Interval}
To obtain bounds at smaller target ratios, keep the same calls and terminal controls.
For $\theta\le r\le\lambda$, restrict the profile $g$ to $[0,r]$ and allow a nonnegative vertical shift.
Such a shift preserves the inequalities $g>L_j$.
The resulting infimum is
\begin{equation}\label{eq:reuse}
 E(r)=\max\left\{g(r),
 \frac{A+rB+w\{\theta\beta+g(r)-g(\theta)\}}{1+w}\right\},
 \qquad\theta\le r\le\lambda.
\end{equation}
For every $\varepsilon>0$, the whole family $\DD(p,E+\varepsilon,[\theta,\lambda])$ holds with the constant in~\Cref{eq:C}.
To check strictness uniformly, use the shift
$H_r=E(r)-g(r)+\varepsilon/2$ and output $z_r=E(r)+\varepsilon$.
The endpoint gap is $\varepsilon/2$, and
\[
 z_r+w[g(\theta)+H_r]-A-rB-w\theta\beta
 \ge(1+w/2)\varepsilon.
\]
\Cref{thm:route} applies uniformly.

A source leaf is affine, and $E$ is the maximum of two affine functions on each affine piece of $g$.
When a new route uses one of these bounds as a call, refine its cell at the partition points of the called profile.
On each resulting piece, adding the affine loss $S$ preserves convexity.
The maximum in~\Cref{eq:route-costs} is therefore attained at an endpoint of a call cell or a partition point of the called profile inside that cell.

\section{The Optimum over Finite Proofs}\label{sec:closure}
We keep the source $\BB_f$ fixed and minimize over all finite repetitions of the retained-set rule.
\begin{definition}[Finite Proof]\label{def:finite-proof}
A \emph{finite proof} is a finite derivation of conditional bounds using the following rules:
\begin{enumerate}
\item Start with a weighted source bound from~\Cref{eq:leaf}.
\item Apply \Cref{thm:route} with previously derived conditional bounds as its calls, including its uniform families.
\item Restrict to a closed ratio interval, increase the density, or increase the continuous exponent function pointwise.
\item Combine finitely many bounds: on a finite closed cover of a compact ratio interval, use a continuous output that dominates a proved family at the same density on each cover interval.
\end{enumerate}
\end{definition}
Taking the maximum of the constants and thresholds validates the last rule.
Let $U_{\rm fin}(p,s)$ be the infimum of the exponents obtained at $(p,s)$ by these finite proofs.

To compare these proofs across all finite depths, first let $U_0(p,s)$ be the infimum over source leaves alone.
It is finite: take $w=1$, $\mu=1/2$, $A>\log2-2\log p$, and $B>b(A)$.
These strict controls give $0\le U_0(p,s)<\infty$.

To treat all possible descents at once, let $u(p,s)$ specify the required exponent for a call at $(p,s)$.
Order these functions pointwise and consider the complete lattice
$\mathcal L=\{u:(0,1)\times(0,1]\to\mathbb R:0\le u\le U_0\}$.

Define $T(u)(p,\lambda)$ as the infimum of $U_0(p,\lambda)$ and all $z>0$ admitted by the following data:
a finite partition $0<\theta=u_0<\cdots<u_m=\lambda$, densities $q_j\in(0,1)$, slopes $d_j>c(q_j)$, and a continuous positive function $g$, affine with slope $d_j$ on each $J_j$ and constant on $[0,\theta]$, such that
\begin{equation}\label{eq:oracle}
 g(s)>u(q_j,s)\qquad(s\in J_j).
\end{equation}
The terminal parameters form a strict control at some row density $0<\pi<p$ and satisfy~\Cref{eq:route-test}.
All data are finite and fixed before $k$ is chosen.

The operator is monotone: decreasing the internal requirements can only add feasible routes.
Its output is nonincreasing in the outer density $p$, since any leaf or route valid at $p$ remains valid at a larger density.
Set
\[
 U_{n+1}=T(U_n),\qquad U_\infty=\inf_{n\ge0}U_n.
\]

We say that $v\in\mathcal L$ is a \emph{lower bound preserved by the rules} if $v\le T(v)$.
In standard order terminology, such a function is a \emph{post-fixed point} of $T$.
Equivalently, every leaf and every route using internal requirements $v$ has output at least $v$.
Here a proof on a relative neighborhood of $s$ means a conditional bound on a compact ratio interval containing that neighborhood.

\begin{theorem}[Finite Proofs and Preserved Lower Bounds]\label{thm:closure}
The sequence $U_n$ decreases to the greatest fixed point $U_\infty$ of $T$ in $\mathcal L$, and
\begin{equation}\label{eq:characterization}
 U_{\rm fin}(p,s)=U_\infty(p,s)
 =\sup_{\substack{v\in\mathcal L\\v\le T(v)}}v(p,s).
\end{equation}
Every $z>U_\infty(p,s)$ has a finite proof with constant output $z$ on a relative neighborhood of $s$, with uniform constants.
\end{theorem}

The order comparison is the classical greatest-fixed-point principle~\cite[Theorem~1]{Tarski}.
Strict inequalities and compactness also identify the limit with finite conditional proofs, including the uniformity required for their reuse.

\begin{proof}
Since $T(U_0)\le U_0$, monotonicity gives $U_{n+1}\le U_n$.
By induction on $n$, we prove both that each ratio section $U_n(p,\cdot)$ is upper semicontinuous and that every $z>U_n(p,s)$ has a finite proof with constant output $z$ on a relative neighborhood of $s$.
Upper semicontinuity means that a strict inequality $U_n(p,s)<a$ persists near $s$.
At level zero, both conclusions follow by choosing a continuous source leaf strictly below $z$.

Suppose the claim holds at level $n$, and let $z>U_{n+1}(p,\lambda)$.
Choose a leaf or a formal route against $U_n$ with output $z'<z$.
A leaf again gives both conclusions by continuity.
For a route, truncate or extend its last affine cell, keeping the density and slope.
Upper semicontinuity preserves $g>U_n(q_m,\cdot)$ near $\lambda$, and the strict endpoint and terminal tests persist by continuity.
Thus the same output $z'$ is feasible for nearby endpoints, proving $U_{n+1}(p,r)<z$ there.

Choose a compact interval of endpoints within this range that contains a relative neighborhood of $\lambda$.
To replace the formal calls by finite proofs, consider the call cells up to the largest endpoint in this neighborhood.
At each point $s$ of a cell with density $q_j$, choose
\[
 U_n(q_j,s)<z_s<g(s).
\]
The induction hypothesis gives a finite proof with constant output $z_s$ near $s$.
Shrink this neighborhood until $z_s<g$ throughout.
Compactness gives a finite subcover of each cell.
Subdivide into closed subcells contained in these neighborhoods and use the corresponding proofs, with the original density and slope.
The calls and their strict inequalities also hold at every shared endpoint.
There are finitely many subcells, so their positive gaps have a common positive lower bound.
The uniform part of \Cref{thm:route} now gives a finite proof with constant output $z$ on the chosen endpoint neighborhood.
This completes the induction.

We next show that $U_\infty$ is fixed.
Monotonicity gives $T(U_\infty)\le U_{n+1}$ for every $n$, hence $T(U_\infty)\le U_\infty$.
For the reverse inequality, take a route feasible against $U_\infty$.
On each closed cell $J_j$, the sets
\[
 O_n=\{s\in J_j:U_n(q_j,s)<g(s)\}
\]
are relatively open, increase with $n$, and cover $J_j$.
Compactness and nesting give a common $N$ with $O_N=J_j$ on every cell.
The route is therefore feasible against $U_N$, so its output is at least $U_{N+1}\ge U_\infty$.
The leaf alternative is also at least $U_\infty$.
Taking infima gives $T(U_\infty)=U_\infty$.

If $v\in\mathcal L$ and $v\le T(v)$, then $v\le U_0$ and
\[
 v\le U_n\quad\Longrightarrow\quad
 v\le T(v)\le T(U_n)=U_{n+1}.
\]
Hence $v\le U_\infty$.
As $U_\infty$ is itself fixed, it is the greatest post-fixed point.

Every finite proof has output at least $U_\infty$, by induction through its rules.
This holds for leaves since $U_0\ge U_\infty$.
A route whose calls are at least $U_\infty$ is feasible against $U_\infty$, so its output is at least $T(U_\infty)=U_\infty$.
The case $\lambda=\theta$ reduces to a leaf.
Restriction, increasing the exponent, and finite patching preserve the comparison.
Increasing the density does too, since $U_\infty$ is nonincreasing in that density.
Thus $U_{\rm fin}\ge U_\infty$.

Conversely, every $z>U_\infty(p,s)$ exceeds $U_n(p,s)$ for some finite $n$.
The local proof constructed above gives $U_{\rm fin}\le U_\infty$ and the claimed uniformity.
\end{proof}

\section{Concave Lower Bounds}\label{sec:barrier}
We construct a function $v\le T(v)$, giving the lower estimate by \Cref{thm:closure}.
Its exponent will be a \emph{concave polygon}: a continuous piecewise affine concave function $F:[0,1]\to\mathbb R$.
Write its partition points and values as
\[
 0=s_0<s_1<\cdots<s_N=1,\qquad F_i=F(s_i),\qquad F_0=0,
\]
and suppose the slopes $d_i=(F_i-F_{i-1})/(s_i-s_{i-1})$ satisfy
$d_1\ge\cdots\ge d_N>0$.
A call at density $q>1-\e^{-d_i}$ requires a route slope greater than $c(q)>d_i$.
This relation suggests using $1-\e^{-d_i}$ as the density up to which $F$ gives a lower bound.
Define
\[
 p_F(s)=1-\e^{-d_i}\quad(s_{i-1}<s\le s_i),\qquad
 v_F(p,s)=\begin{cases}F(s),&p\le p_F(s),\\0,&p>p_F(s).\end{cases}
\]
At each $s_i>0$, $p_F$ uses the slope of the segment ending there.

\begin{proposition}[A Preserved Lower Bound]\label{prop:barrier}
Suppose every strict terminal control of row density at most $p_F(\lambda)$ satisfies
\begin{equation}\label{eq:barrier-terminal}
 F(\lambda)+wF(\theta)
 \le A+\lambda B+w\theta\beta
 \qquad(0\le\theta\le\lambda\le1,\ \lambda>0).
\end{equation}
Then $v_F\le U_\infty$.
If $d_N\ge\log2$, in particular $F(1)\le U_\infty(1/2,1)$.
\end{proposition}

\begin{proof}
Taking $\theta=\lambda$ in~\Cref{eq:barrier-terminal} shows that every source leaf at density $p\le p_F(\lambda)$ has exponent at least $F(\lambda)$.
Thus $0\le v_F\le U_0$.

Consider a route feasible against $v_F$ at density $p\le p_F(\lambda)$, with output $z<F(\lambda)$.
Its endpoint satisfies $g(\lambda)<F(\lambda)$.
Refine the call partition at the points $s_i$.
On a resulting cell $[a,b]$, let $\ell_j$ and $d_i$ be the slopes of $g$ and $F$.
If $g(b)<F(b)$, the closed call condition forces $q_j>p_F(b)=1-\e^{-d_i}$, since $p_F$ uses the left slope at $b$.
Consequently $\ell_j>c(q_j)>d_i$, and
\[
 g(a)-F(a)=g(b)-F(b)-(\ell_j-d_i)(b-a)<0.
\]
Starting at $\lambda$ and applying this implication backward through the finite partition gives $g(\theta)<F(\theta)$.
The row density satisfies $\pi<p\le p_F(\lambda)$, so~\Cref{eq:barrier-terminal} now contradicts
$z+wg(\theta)>A+\lambda B+w\theta\beta$.
Every route and leaf therefore has output at least $v_F$, giving $v_F\le T(v_F)$.
Apply \Cref{thm:closure}; the final assertion follows from $p_F(1)\ge1/2$.
\end{proof}

\subsection{Elimination of the Terminal Controls}
At each positive partition point, we maximize over the stopping ratio and combine the terminal inequalities with the supporting-line formula.
This gives a test in the single parameter $\mu$.
For $0<\mu<1$, write $\tau=-\log(1-\mu)$ and $\beta=-\log\mu$, and define
\begin{equation}\label{eq:Psi}
 \Psi(\lambda,z,\kappa,\mu)
 =\frac{1-\mu}{\kappa}
 \left[z-\lambda\widehat f\!\left(\frac{1-\kappa}{\lambda}\right)\right],
 \qquad \lambda>0,\quad 0<\kappa<1.
\end{equation}

\begin{theorem}[Finite Polygon Criterion]\label{thm:polygon}
Suppose $F$ is as above and
\begin{equation}\label{eq:node-bounds}
 0<F_i<\min\{f(s_i),h(s_i)\}\qquad(1\le i\le N).
\end{equation}
For each $\mu\in(0,1)$ define
\begin{equation}\label{eq:D}
 D_i(\beta)=\max_{0\le j\le i}(F_j-s_j\beta),\qquad
 \kappa_i=D_i(\beta)/\tau,\qquad
 P_i=-\log(1-\e^{-d_i}).
\end{equation}
If, for every $i$ and every $\mu$ with $D_i(\beta)>0$,
\begin{equation}\label{eq:node-test}
 \Psi(s_i,F_i,\kappa_i,\mu)\le P_i,
\end{equation}
then $v_F\le U_\infty$.
\end{theorem}

\begin{proof}
At endpoint $s_i$, the function $F(\theta)-\theta\beta$ is affine between consecutive partition points, so
\[
 \max_{0\le\theta\le s_i}\{F(\theta)-\theta\beta\}=D_i(\beta).
\]
Concavity of $h$ and \Cref{eq:node-bounds} put $F(\theta)$ strictly below $h(\theta)$ for $\theta>0$.
The minimum of $\tau+\theta\beta$ over $0<\mu<1$ occurs at $\mu=\theta/(1+\theta)$ and equals $h(\theta)$.
Thus
\[
 F(\theta)<h(\theta)\le\tau+\theta\beta\qquad(\theta>0).
\]
The term at $\theta=0$ is zero, so $0\le D_i(\beta)<\tau$.

Fix a strict control at row density $\pi\le1-\e^{-d_i}$ and put $P_\pi=-\log\pi\ge P_i$.
If the terminal inequality fails at $\lambda=s_i$ for some stopping ratio, then
\[
 wD_i>A+s_iB-F_i>0,\qquad
 A>w\tau+\frac{P_\pi}{1-\mu}.
\]
Here $A+s_iB-F_i>0$ because $A+s_iB\ge f(s_i)>F_i$.
In particular $0<\kappa_i=D_i/\tau<1$.
Multiplying the second inequality by $\kappa_i$ and using the first gives
\[
 (1-\kappa_i)A+s_iB
 <F_i-\frac{\kappa_i P_\pi}{1-\mu}.
\]
Since $B>b(A)$, the supporting-line formula \Cref{eq:dual} gives
\[
 s_i\widehat f\!\left(\frac{1-\kappa_i}{s_i}\right)
 \le(1-\kappa_i)A+s_iB.
\]
Combining these inequalities yields
$P_\pi<\Psi(s_i,F_i,\kappa_i,\mu)\le P_i$, a contradiction.
Thus \Cref{eq:barrier-terminal} holds for every stopping ratio at each positive partition point.

Now fix controls at row density at most $1-\e^{-d_i}$ and let $\lambda\in[s_{i-1},s_i]$.
Since $F$ is affine on this interval,
\[
 \max_{0\le\theta\le\lambda}\{F(\theta)-\theta\beta\}
 =\max\{D_{i-1}(\beta),\ F(\lambda)-\lambda\beta\},
\]
where $D_0=0$.
The largest terminal difference over all stopping ratios is therefore
\[
 F(\lambda)-A-\lambda B
 +w\max\{D_{i-1}(\beta),\ F(\lambda)-\lambda\beta\}.
\]
This is convex in $\lambda$, being affine plus a positive multiple of a maximum of two affine functions.
It is nonpositive at $s_i$ by the inequality already proved there.
The same holds at $s_{i-1}>0$ because the left slope of $F$ there is at least $d_i$; at zero the value is $-A<0$.
Convexity gives \Cref{eq:barrier-terminal} throughout the interval.
Apply \Cref{prop:barrier}.
\end{proof}

\subsection{The Full Control Interval}
For the supplied polygon, the remaining variable is $\tau\in(0,\infty)$, with
$\mu=1-\e^{-\tau}$ and $\beta(\tau)=-\log(1-\e^{-\tau})$.
We verify the criterion analytically near zero and infinity, leaving a compact interval for the finite computation.
Let $L=1/10$ and $H=10$.
The supplied data satisfy
\begin{equation}\label{eq:tails}
 \beta(L)>d_1,\qquad
 \beta(H)<F_N/2,\qquad F_N<H,\qquad
 2H\e^{-H}<P_1.
\end{equation}
For $\tau\le L$, concavity gives $F_j\le d_1s_j<\beta s_j$ for $j>0$, so $D_i=0$.
For $\tau\ge H$, concavity gives $F_i/s_i\ge F_N$ and hence
$D_i\ge F_i-s_i\beta\ge F_i/2$.
Also $D_i\le F_N<\tau$.
The source term in~\Cref{eq:Psi} is nonnegative, so
\[
 \Psi(s_i,F_i,\kappa_i,\mu)
 \le\e^{-\tau}\frac{F_i\tau}{D_i}
 \le2\tau\e^{-\tau}\le2H\e^{-H}<P_1\le P_i.
\]
Thus only the closed interval $[L,H]$ requires subdivision.

On each control cell, outward interval arithmetic first encloses $\beta$.
The function $D_i$ is nonincreasing in $\beta$, so its two endpoint evaluations enclose the finite maximum in~\Cref{eq:D}.
The source $\widehat f$ increases, so its endpoint values enclose the source term in~\Cref{eq:Psi}.
A cell is accepted if it proves $D_i=0$ throughout, or a nonpositive numerator in~\Cref{eq:Psi} wherever $D_i>0$.
Points with $D_i=0$ already satisfy the terminal condition, since $F_i<A+s_iB$.
Otherwise the checker requires a positive lower bound for $\kappa_i$ before division, and accepts the cell only if it proves $P_i-\Psi>0$.
All other cells are subdivided.
The checker verifies an exact closed cover of $[L,H]$ for every positive partition point.
\Cref{thm:polygon} then covers all weights, source supports, stopping ratios, endpoint ratios, and finite depths.

\section{The Diagonal Bounds}\label{sec:certificate}
We now evaluate the two sides of~\Cref{eq:characterization}.
The repository specifies the source, a finite upper proof, and a rational lower polygon, with the following complete interval covers.
\begin{center}
\begin{tabular}{lrr}
\toprule
Construction & Size & Checked Cells\\
\midrule
Source profile & 2,279 cubic pieces & 12,319\\
Upper proof & 2,932 routes; depth 27 & 1,065,279\\
Lower polygon & 2,048 affine pieces & 853,082\\
\bottomrule
\end{tabular}
\end{center}

The source inequalities are evaluated at 224 fractional binary bits, and the upper and lower inequalities at both 224 and 320 bits.
An independent integer-prefix computation reproduces every route height and loss from the checked primitive enclosures.
The repository contains the four rational inputs in \path{data/}, the Python certificate checker in \path{proof/}, and the Lean formalization in \path{lean/}.
The Python checker uses only the standard library. Running \texttt{python proof/replay.py -{}-workers 16} from the repository root writes the results and input and code hashes to \path{proof/build/replay/}.
The root \path{README.md} gives the build and reproduction instructions for both projects.

\subsection{Exact Interval Arithmetic}\label{sec:arithmetic}
Transcendental quantities are enclosed by intervals whose endpoints are integers divided by a fixed power of two. Rational inputs and arithmetic operations are rounded outward. For logarithms, reduce to \(y\in[1,2]\) and put \(u=(y-1)/(y+1)\in[0,1/3]\). Then
\begin{equation}\label{eq:logseries}
 \log y=2\sum_{j=0}^{m-1}\frac{u^{2j+1}}{2j+1}+E_m,
 \qquad 0\le E_m\le\frac{3u^{2m+1}}{2m+1}.
\end{equation}
The tail is at most \(2u^{2m+1}/((2m+1)(1-u^2))\). The same formula encloses \(\log2\) for reversing the scaling.

For the exponential, put $S_m(z)=\sum_{j=0}^m z^j/j!$ for integers $m\ge0$.
Successive terms after degree $m$ have ratio at most $z/(m+2)$, so a geometric sum gives
\begin{equation}\label{eq:expseries}
 S_m(z)<\e^z\le S_m(z)+
 \frac{z^{m+1}}{(m+1)!}\frac1{1-z/(m+2)}
 \qquad(0<z<m+2).
\end{equation}
At zero the sum is exact.
The interval backend reduces a nonnegative argument to $0\le y\le1/8$, where this tail is at most $2y^{m+1}/(m+1)!$.
Repeated squaring reverses the scaling; negative arguments use outward reciprocals.
A logarithm or division is accepted only after its domain is checked.
For the final exponential bound, a separate rational calculation applies~\Cref{eq:expseries} directly with $m=140$ and compares the upper enclosure with $3.69507$.

\subsection{A Finite Upper Proof}
The file \path{data/upper.json} contains an ordered list of routes.
A call names either a weighted source bound or an earlier route with a terminal control.
All 2,932 routes are reachable from the root.
The density, height, and output margins are all $10^{-12}$; rounding is upward to multiples of $10^{-12}$.
The source pairs use the inward factor $1-10^{-9}$, and we put $\rho=1-2\cdot10^{-5}$.

For each control, let $\pi$ be its leaf density when it defines a source bound, and $\rho p$ when it ends a route at outer density $p$.
The checker computes
\[
 x_{\rm cap}=(1-\mu)^w\pi^{1/(1-\mu)}
\]
and an explicit boundary pair $(\xi,\zeta)$ with $\xi\ge x_{\rm cap}$.
It uses $x=(1-10^{-9})x_{\rm cap}$ and $y=(1-10^{-9})\zeta$.
Both coordinates are strictly below a supporting pair, and $x<x_{\rm cap}$, so the control is strict.
Source supports are given explicitly in the data and checked by interval inequalities.
For a source leaf, the coefficients $A/(1+w)$ and $(B+w\beta)/(1+w)$ are rounded upward, giving a rational affine upper bound for~\Cref{eq:leaf}.

For a reused child with stored profile $g$ and stopping ratio $\theta$, upward enclosures of its terminal parameters give rational coefficients
\[
 a^+\ge\frac{A+w\theta\beta}{1+w},\qquad
 b^+\ge\frac{B}{1+w},\qquad t^+\ge\frac{w}{1+w}.
\]
Its called family is
\[
 \max\{g(s),\ a^++b^+s+t^+(g(s)-g(\theta))\}+10^{-12}.
\]
This dominates~\Cref{eq:reuse} with a strict allowance; rounding $t^+$ upward is valid because $g(s)-g(\theta)\ge0$.

With these calls established, the parent route can be evaluated.
Each slope is an upward-rounded enclosure of $c(q_j)+10^{-12}$, so the sums defining $S$ are rational.
The endpoint rule in~\Cref{eq:route-costs} gives $M$ by checking the parent cell endpoints and the child profiles' partition points inside each cell.
With $Q=10^{12}$, the parent stores
\[
 H=\frac{\lceil Q(M+10^{-12})\rceil}{Q},\qquad g=H-S.
\]
The checker also verifies the full partitions, density implications, child-domain containment, and strict ordering of child indices.
Induction through the ordered list and \Cref{thm:route} prove all its conditional families.

The root has
\[
 p=\tfrac12,\qquad\pi=\tfrac12-10^{-5},\qquad
 \theta=\tfrac14,\qquad\lambda=1,\qquad w=1,\qquad\mu=\tfrac14.
\]
It has 65 cells and uses the left source support in \path{data/upper.json}, at $t\approx0.383$.
Exact reconstruction gives $g(1)>1.3$, while the second term in~\Cref{eq:reuse} is less than $1.2$.
Adding the output margin yields the constant conditional exponent
\[
 z_+=g(1)+10^{-12}=\zp.
\]
Thus $U_\infty(1/2,1)\le z_+$.

\subsection{The Lower Polygon}
The file \path{data/lower.json} specifies the 2,049 rational partition points and values of $F$.
Exact rational comparisons establish positive decreasing slopes and $F(0)=0$.
Interval comparisons establish~\Cref{eq:node-bounds}, the tail bounds in~\Cref{eq:tails}, and $d_N>\log2$.
The 853,082 closed control cells verify~\Cref{eq:node-test}.
\Cref{thm:polygon} and \Cref{prop:barrier} therefore give
\[
 z_-:=F(1)=\zm\le U_\infty(1/2,1).
\]

The resulting interval for the finite-proof optimum has width less than $0.0016$.

\begin{proof}[Proof of \Cref{thm:main}]
The finite upper proof and the preserved polygon give~\Cref{eq:bracket} by \Cref{thm:closure}.
In an arbitrary two-coloring, interchange the colors so that red has density at least $1/2$ and apply the root conditional bound.
Its constant is
\[
 \frac{3(1-\pi)}{1/2-\pi}=150003,
\]
which proves~\Cref{eq:main}.
An independent rational Taylor enclosure gives $\e^{z_+}<3.69507$.
This strict gap absorbs the fixed factor and the ceiling for sufficiently large $k$.
\end{proof}

\clearpage
\appendix

\section{The Finite-Host Weighted Inequality}\label{sec:weighted-proof}
We prove the estimate used in \Cref{prop:candidate}.
Its stopping assumption concerns only proper subsets of the host.
This lets the same estimate prove the source bound by a smallest-counterexample argument in \Cref{sec:source-proof}.

\begingroup
\postdisplaypenalty=10000
\begin{lemma}[Finite-Host Weighted Candidate Bound]\label{lem:local}
Fix \(w>0\) and parameters in \((0,1)\) satisfying
\[
 x_0<\widehat x,\qquad y_0<\widehat y,\qquad
 x_0<(1-\mu_0)^w p^{1/(1-\mu_0)}.
\]
There is \(L_0\), depending only on these parameters, with the following property, uniformly in \(C\ge1\) and the finite host coloring \(H\). Fix \(\ell\ge L_0\), and suppose that for every integer \(a\ge1\), every proper subset \(Z\subsetneq V(H)\) of size
\begin{equation}\label{eq:local-stop}
 |Z|\ge C\widehat x^{-a}\widehat y^{-\ell}
\end{equation}
contains a red \(K_a\) or a blue \(K_\ell\). Then, for all \(k,t\ge1\), every candidate in \(H\) satisfying
\begin{equation}\label{eq:local-size}
 d_R(X,Y)\ge p,\qquad
 |X|^w|Y|\ge C^{w+1}x_0^{-k}y_0^{-\ell}\mu_0^{-wt}
\end{equation}
is \((k,\ell,t)\)-good.
\end{lemma}
\endgroup
\begin{proof}
We adapt the candidate induction of~\cite[Lemma~12]{GNNW} to unequal weights. After removing rows with too few red neighbors, we either pass to a common blue neighborhood or choose a vertex whose red or blue child satisfies the induction hypothesis. A power of the excess red density controls the size loss in these steps.

Choose a fixed integer \(r>\max\{1,w\}\) so large that \(p^{1/r}>\mu_0\) and
\[
 x_0<(p^{1/r}-\mu_0)^r(1-\mu_0)^{w-r}.
\]
The logarithm of the right side tends to \(w\log(1-\mu_0)+\log p/(1-\mu_0)\). The function
\[
 x_0^{1/r}(1-q)^{1-w/r}+\mu_0^{w/r}q^{1-w/r}
\]
increases until \(q=\mu_0/(\mu_0+x_0^{1/w})>\mu_0\), and is less than \(p^{1/r}\) at \(q=\mu_0\). By continuity choose
\[
 x_0<x<\widehat x,\quad y_0<y<\widehat y,\quad
 \mu_0<\mu<\beta<1,\quad0<\delta<p/2,
\]
and \(\gamma>0\) so that, with \(\bar p=p-\delta\),
\begin{equation}\label{eq:scalar}
 \max_{0\le q\le\beta}
 \{x^{1/r}(1-q)^{1-w/r}+\mu^{w/r}q^{1-w/r}\}
 \le(1-\gamma)\bar p^{1/r}.
\end{equation}
All these choices precede \(C,H,k,\ell,t\).

Put \(n=k+t\), \(\delta_n=\delta/n\), and induct on \(n\), for fixed sufficiently large \(\ell\), using the stronger hypothesis
\begin{equation}\label{eq:moment}
 d_R(X,Y)\ge p-\delta_n,\qquad
 Q_n=(d_R(X,Y)-p+\delta_n)^r|X|^w|Y|
 \ge C^{w+1}x^{-k}y^{-\ell}\mu^{-wt}.
\end{equation}
After canceling \(C^{w+1}\), the conditions in \Cref{eq:local-size} imply this once \(\ell\) is large, uniformly in \(k,t\), because
\[
 (\delta/n)^r(x/x_0)^k(y/y_0)^\ell(\mu/\mu_0)^{wt}\ge1.
\]
The factors involving \(k,t\) grow exponentially in \(n\), so their ratio to \(n^r\) has a positive infimum; increasing \(\ell\) makes the inequality hold. If \(k=1\) or \(t=1\), any vertex of \(X\) suffices. Henceforth \(k,t\ge2\), and the density excess lies in \((0,1)\).

\emph{Regularize and stop if \(Y\) is large.}
Delete rows of red degree less than \((p-\delta_n)|Y|\). The excess edge count \(E=e_R(X,Y)-(p-\delta_n)|X||Y|\) increases, and
\[
 Q_n=E^r|X|^{w-r}|Y|^{1-r}
\]
does not decrease, since \(r>w\). Positive excess prevents deletion of the last row. Every nonempty \(T\subseteq X\) now has \(d_R(T,Y)\ge p-\delta_n\).
If \(|Y|\ge C\widehat x^{-k}\widehat y^{-\ell}\), apply the stopping assumption with \(a=k\); \(Y\) is proper because \(X\ne\varnothing\). Otherwise \(Q_n\le|X|^w|Y|\) implies
\begin{equation}\label{eq:largeX}
 |X|^w>C^w(\widehat x/x)^k(\widehat y/y)^\ell\mu^{-wt},
 \qquad |X|\ge C\exp(c(k+\ell+t))
\end{equation}
for a fixed \(c>0\).

\emph{Find a common blue neighborhood or a suitable vertex.}
Define
\[
 b=\left\lceil\frac{2r\log n+r\log(1/\delta)+w\log2}
 {w\log(\beta/\mu)}\right\rceil,
 \qquad m=\lceil10\beta^{-1}b^2\rceil,
\]
and \(W=\{v\in X:\deg_B(v,X)\ge\beta|X|\}\). Put \(\xi_n=\delta/(2n^2)\).
We choose \(L_0\) once so that
\[
 |X|\ge5m^2,\qquad
 R(k,m)<\bar p\xi_n|X|,\qquad
 \frac{2n^2}{\delta|X|}<\gamma.
\]
To justify a uniform choice, note that \(b=O(\log n)\), \(m=O((\log n)^2)\), and
\(R(k,m)\le\binom{k+m-2}{m-1}\le k^{m-1}\); the last inequality counts nondecreasing tuples in \(\{1,\ldots,k\}^{m-1}\).
The logarithm of each required lower bound on \(|X|\) is therefore \(O((\log n)^3)\).
Subtracting \(cn\) leaves a function bounded above in \(n\), so \Cref{eq:largeX} supplies all three bounds for one \(L_0\), independently of \(C\ge1,k,t\).
We use the common-neighborhood averaging argument of~\cite[Lemma~9]{GNNW}.
When \(|W|\ge R(k,m)\), either a red \(K_k\) finishes the proof, or \(W\) contains a blue \(m\)-clique \(U\). Put \(N_X=|X|\). The blue density \(\sigma\) between \(U\) and \(X\setminus U\) satisfies
\[
 \sigma\ge\frac{\beta N_X-m}{N_X-m}
 =\beta-\frac{(1-\beta)m}{N_X-m}\ge\beta-\frac1{4m}.
\]
Choose a uniform \(b\)-subset of \(U\), and let $T$ be its common blue neighborhood in $X\setminus U$. Convexity of the piecewise-linear extension of \(j\mapsto\binom jb\) on nonnegative integers gives
\[
 \mathbb E|T|\ge(N_X-m)
 \frac{\binom{\lfloor\sigma m\rfloor}{b}}{\binom mb}
 \ge(N_X-m)(\sigma-b/m)^b.
\]
The second inequality follows factor by factor from the binomial product. Since \(m\ge10\beta^{-1}b^2\),
\[
 \sigma-b/m\ge\beta\left(1-\frac1{8b}\right)>0,
 \qquad N_X-m\ge\frac45N_X.
\]
Bernoulli's inequality gives \(\mathbb E|T|\ge(7/10)\beta^bN_X\). Thus some blue \(b\)-clique \(S\subseteq U\) has a disjoint common blue neighborhood \(T\) with \(|T|\ge\beta^b|X|/2\).

If $b\ge t$, then $S$ already contains the required blue clique. For $b<t$, we have
\[
 \delta_{n-b}-\delta_n\ge\delta/n^2,\qquad
 Q_{n-b}(T,Y)\ge(\delta/n^2)^r2^{-w}\beta^{wb}|X|^w|Y|
 \ge C^{w+1}x^{-k}y^{-\ell}\mu^{-w(t-b)}.
\]
The last inequality uses \(Q_n\le|X|^w|Y|\) and \((\beta/\mu)^{wb}\ge2^w(n^2/\delta)^r\), which follows from the choice of \(b\). Induction applies; a blue \(K_{t-b}\) in \(T\) joins \(S\) to form a blue \(K_t\), and either of the other two outcomes already suffices.

We may now assume \(|W|<R(k,m)<\bar p\xi_n|X|\). Put \(Y_v=N_R(v)\cap Y\), \(d=d_R(X,Y)\), and \(E_0=e_R(X,Y)\). Row regularization gives \(|Y_v|\ge\bar p|Y|>0\) for every \(v\in X\). Counting by vertices in \(Y\) and applying Cauchy--Schwarz gives
\[
 \sum_{v\in X}d_R(X,Y_v)|Y_v|
 =\frac1{|X|}\sum_{y\in Y}\deg_R(y,X)^2\ge dE_0.
\]
Removing the terms indexed by \(W\) loses at most \(|W||Y|\), so
\[
 \sum_{v\notin W}d_R(X,Y_v)|Y_v|
 \ge dE_0-|W||Y|
 >(d-\xi_n)E_0
 \ge(d-\xi_n)\sum_{v\notin W}|Y_v|.
\]
Here \(E_0\ge\bar p|X||Y|\) and \(d-\xi_n>0\). Hence some \(v\notin W\) satisfies \(d_R(X,Y_v)\ge d-\xi_n\).

\emph{Pass to a red or blue child.}
Set \(Y'=Y_v\), \(X_R=N_R(v)\cap X\), \(X_B=N_B(v)\cap X\), and
\[
 a=|X_R|/|X|,\quad q=|X_B|/|X|,\quad
 \alpha=d_R(X,Y')-p+\delta_{n-1}.
\]
Then \(a+q\le1\), \(q\le\beta\), and \(\alpha\ge d-p+\delta_n+\delta/(2n^2)>0\).
For \(\chi\in\{R,B\}\), set \(\alpha_\chi=d_R(X_\chi,Y')-p+\delta_{n-1}\) when \(X_\chi\ne\varnothing\), and set \(\alpha_\chi=0\) otherwise. Write \(u_+=\max\{u,0\}\) and
\[
 Q_\chi=\pos{\alpha_\chi}^{r}|X_\chi|^w|Y'|.
\]
Splitting \(X\) into its two children and the root gives
\begin{equation}\label{eq:split}
 1\le a\frac{\pos{\alpha_R}}\alpha
       +q\frac{\pos{\alpha_B}}\alpha+\frac1{\alpha|X|}.
\end{equation}
The root's contribution is \((1-p+\delta_{n-1})/(\alpha|X|)\le1/(\alpha|X|)\).
We claim that
\[
 Q_R\ge xQ_n\qquad\text{or}\qquad Q_B\ge\mu^wQ_n.
\]
If both inequalities failed, then \(Q_n\le\alpha^r|X|^w|Y|\), \(|Y'|\ge\bar p|Y|\), and \Cref{eq:split} would give
\[
 1\le\bar p^{-1/r}
 \{x^{1/r}a^{1-w/r}+\mu^{w/r}q^{1-w/r}\}
 +\frac1{\alpha|X|}
 \le1-\gamma+\frac{2n^2}{\delta|X|}<1.
\]
The second inequality uses \(a\le1-q\) and \Cref{eq:scalar}; the last is one of our three size bounds.
Thus one child has a positive value of \(Q_{n-1}\) meeting \Cref{eq:moment}, with \(k\) reduced by one in the red child or \(t\) reduced by one in the blue child. Induction applies. The root extends a red \(K_{k-1}\) in \(X_R\cup Y'\), or a blue \(K_{t-1}\) in \(X_B\), respectively. Every other good outcome already gives a required clique.
\end{proof}

\begin{corollary}[Local Density Bound]\label{cor:local-density}
Under the stopping assumption of \Cref{lem:local}, fix $w>0$ and $\mu,p,x,y\in(0,1)$ with
$x<\widehat x$, $y<\widehat y$, and $x<(1-\mu)^wp^{1/(1-\mu)}$.
If a host has red density at least $p$ and order
\begin{equation}\label{eq:local-density}
 N\ge Cx^{-k/(w+1)}(y\mu^w)^{-\ell/(w+1)},
\end{equation}
then it contains a red $K_k$ or a blue $K_\ell$ for all sufficiently large $\ell$, with a threshold independent of $C\ge1$ and $k$.
\end{corollary}
\begin{proof}
Choose $y<y_1<\widehat y$.
A uniformly chosen balanced partition has expected cross-density equal to the host density, so some disjoint $X,Y$, each of size at least $N/3$, satisfy $d_R(X,Y)\ge p$.
For large $\ell$, $(y_1/y)^\ell\ge3^{w+1}$.
The order bound therefore supplies~\Cref{eq:local-size} at $(x,y_1)$ with $t=\ell$, and \Cref{lem:local} applies.
\end{proof}

\section{Admissibility of the Source}\label{sec:source-proof}
We first give a criterion for a concave Ramsey profile, then verify it for the specified function $f$.
For a continuous $F$ on $[0,1]$ with $F(0)=0$, define
\[
 G_F(a,b)=\max(a,b)F\!\left(\frac{\min(a,b)}{\max(a,b)}\right),
\]
\[
 \BB_F=\{(x,y)\in(0,1)^2:
 G_F(a,b)\le-a\log x-b\log y\text{ for all }a,b>0\}.
\]
Thus $\BB_F$ consists of the exponential bounds that dominate the profile in both coordinates.
For $F=f$, this agrees with the source pairs in \Cref{sec:inputs}, since $G_f(a,b)=b\widehat f(a/b)$.
In the source criterion, $X,Y$ denote scalar support parameters.

\begin{lemma}[Source Criterion]\label{lem:source}
Suppose \(F\) is continuous on \([0,1]\), \(F(0)=0\), is \(C^1\) and concave on \((0,1]\), and has \(F'>0\). If for every \(s\in(0,1]\) there are \(w>0\), \(\mu,Y\in(0,1)\) with
\[
 p_0(s)=1-\e^{-F'(s)},\quad
 X=(1-\mu)^wp_0(s)^{1/(1-\mu)},\quad (X,Y)\in\BB_F,
\]
and
\begin{equation}\label{eq:source-slack}
 (w+1)F(s)+\log X+s\log Y+ws\log\mu>0,
\end{equation}
then, for every \(\varepsilon>0\), some \(C_\varepsilon\ge1\) satisfies
\begin{equation}\label{eq:source-finite}
 R(a,b)\le\left\lceil C_\varepsilon
 \exp\bigl(G_F(a,b)+\varepsilon(a+b)\bigr)\right\rceil
 \qquad(a,b\in\mathbb Z_{\ge1}).
\end{equation}

\end{lemma}
\begin{proof}
Fix \(\varepsilon>0\). We first choose finitely many local bounds, then apply one of them to a smallest counterexample.
Since \(F(0)=0\) and \(F'>0\), we have \(F\ge0\).
Choose \(0<\alpha<1\) with \(h(\alpha)<\varepsilon/2\); the classical bound will handle smaller ratios.
At each \(s_0\in[\alpha,1]\), choose \(w,\mu,X,Y\) as in the hypothesis and put \(\bar x=\e^{-\varepsilon}X\), \(\bar y=\e^{-\varepsilon}Y\).
Scaling the controls this way and replacing \(F(s_0)\) by \(F(s_0)+\varepsilon(1+s_0)\) increases the slack by \(w\varepsilon(1+s_0)\).
We may therefore choose
\[
 0<x<\widehat x<\bar x,\qquad
 0<y<\widehat y<\bar y,\qquad 0<p<p_0(s_0),
\]
close enough to \((\bar x,\bar y,p_0(s_0))\) that \(x<(1-\mu)^wp^{1/(1-\mu)}\) and
\[
 H(t):=-\frac{\log x+t\log(y\mu^w)}{w+1},\qquad
 H(s_0)<F(s_0)+\varepsilon(1+s_0).
\]
Continuity gives a neighborhood on which this inequality holds and \(p_0(s)>p+2\eta_{s_0}\) for some \(\eta_{s_0}>0\).
A finite subcover gives \(\eta>0\) and a common threshold \(L\ge2\) for \Cref{cor:local-density}, both independent of \(C\).
At every \(s\in[\alpha,1]\), one selected local bound satisfies \(H(s)<F(s)+\varepsilon(1+s)\) and \(p_0(s)>p+2\eta\).
Every selected pair also satisfies
\begin{equation}\label{eq:majorant}
 -a\log\widehat x-b\log\widehat y
 \ge G_F(a,b)+\varepsilon(a+b)\qquad(a,b>0),
\end{equation}
by the support condition and the scaling by \(\e^{-\varepsilon}\).

Choose an integer \(K\) with \(\alpha K\ge L\), and then \(C\ge\max\{3,2/\eta,4^K\}\). Suppose there is a coloring with no required clique at an integer order
\[
 N\ge C\exp\bigl(G_F(a,b)+\varepsilon(a+b)\bigr).
\]
Choose one of smallest order among all parameter pairs and colorings. Interchange the colors if needed and write its targets as \(k\ge\ell\). The classical bound \cite{ES} excludes \(k<K\). It also excludes \(\ell/k\le\alpha\), since
\[
 \log R(k,\ell)\le kh(\ell/k)\le kh(\alpha)<\varepsilon k/2
 \le kF(\ell/k)+\varepsilon(k+\ell)+\log C.
\]
Thus \(s=\ell/k\in[\alpha,1]\) and \(\ell\ge L\).

For every integer \(a\ge1\), minimality and \Cref{eq:majorant} give a red \(K_a\) or a blue \(K_\ell\) in each proper subset satisfying~\Cref{eq:local-stop}.
Thus the stopping assumption of \Cref{cor:local-density} holds for the local bound selected at \(s\).
If the host's red density is at least \(p_0(s)-\eta>p\), then \(N>C\e^{kH(s)}\), and \Cref{cor:local-density} gives a required clique.

Otherwise its blue density exceeds \(q=\e^{-F'(s)}+\eta<1\). A blue neighborhood \(Z\) has size at least \(q(N-1)>N\e^{-F'(s)}\), since \(N\eta>1\). Concavity gives
\[
 k\bigl(F(s)-F(s-1/k)\bigr)\ge F'(s).
\]
Consequently
\[
 |Z|>N\e^{-F'(s)}
 \ge \e^\varepsilon C\exp\bigl(kF(s-1/k)+\varepsilon(k+\ell-1)\bigr).
\]
The proper subset \(Z\) therefore has order above the proposed \((k,\ell-1)\) threshold. Minimality gives a red \(K_k\) or a blue \(K_{\ell-1}\); in the latter case, add the chosen vertex. Either outcome contradicts the choice of the host.

Hence \Cref{eq:source-finite} holds. For \(1\le b\le a\), its logarithm is at most \(aF(b/a)+2\varepsilon a+\log(2C_\varepsilon)\). Since \(\varepsilon\) is arbitrary, the error is uniform in \(b\).
\end{proof}

\subsection{The Piecewise Cubic Profile}
We now verify the criterion for $f$ in \Cref{eq:source-f}. The function $P$ has 2,279 rational cubic pieces.
The file \path{data/source.json} gives the rational partition points, values, and derivatives of \(P\). On a cell \([a,b]\), set \(u=t-a\), \(\delta=b-a\), and \(\Delta=P(b)-P(a)\). Its exact polynomial is
\[
 P(a+u)=P(a)+P'(a)u
 +\frac{3\Delta/\delta-2P'(a)-P'(b)}\delta u^2
 +\frac{P'(a)+P'(b)-2\Delta/\delta}{\delta^2}u^3.
\]
Exact endpoint identities ensure that \(P\) is globally \(C^1\). For \(t>0\) in the interior of a piece, with \(j(t)=t\e^{-t}P(t)\),
\[
 j''(t)=\e^{-t}\bigl((t-2)P+(2-2t)P'+tP''\bigr),
 \qquad f''(t)=-\frac1{t(1+t)}+j''(t).
\]
The cover in \path{data/source_cells.json} proves \(1-t(1+t)j''(t)>0\) on each closed piece, using that piece's endpoint derivatives, together with \(f'(1)>0\) and \(2f'(1)>f(1)\).
Consequently $f(0)=0$, $f$ is continuous on $[0,1]$, $C^1$ on $(0,1]$, strictly concave and increasing.

To describe the boundary of $\BB_f$, put $x_f(t)=\e^{-f'(t)}$ and $y_f(t)=\e^{tf'(t)-f(t)}$.
Their endpoint limits are
\[
 x_f(t)=\frac{t}{1+t}\e^{-j'(t)}\longrightarrow0,\qquad
 y_f(t)=\frac1{1+t}\e^{tj'(t)-j(t)}\longrightarrow1.
\]

Strict concavity makes $x_f$ increasing and $y_f$ decreasing.
Also $x_f(t)<y_f(t)$: the function $(1+t)f'(t)-f(t)$ decreases to its positive value at $t=1$.
Writing $x_1=x_f(1)$ and $y_1=y_f(1)$, the upper boundary is
\begin{equation}\label{eq:Y}
 Y_f(x)=
 \begin{cases}
 y_f(t),&0<x<x_1,\quad x_f(t)=x,\\
 \e^{-f(1)}/x,&x_1\le x\le y_1,\\
 x_f(t),&y_1<x<1,\quad y_f(t)=x.
 \end{cases}
\end{equation}
These are the supporting lines of $\widehat f$ on its two smooth branches and at its corner.
Thus $b(A)=-\log Y_f(\e^{-A})$.

\subsection{The Source Inequalities}\label{sec:source-verification}
The file \path{data/source_cells.json} supplies a cover of 12,319 closed cells and rational controls \(r,w>0\) on each cell.
Set \(\mu=tr\), take \(X\) as in \Cref{lem:source}, and choose \(Y=Y_f(X)\).
With \(\Delta_f\) denoting the left side of \Cref{eq:source-slack}, the cell inequalities give
\[
 \Delta_f(t)/t>10^{-8}\qquad(0<t\le1).
\]
The cell inequalities also establish \(0<\mu,p_0<1\) for positive \(t\). Since \Cref{lem:source} requires controls at each ratio, they may be chosen separately on each cell.

Near zero, we cancel the logarithmic terms before evaluating the slack. Define
\[
 L_+(u)=\frac{\log(1+u)}u,\qquad L_-(u)=\frac{-\log(1-u)}u,
 \qquad L_+(0)=L_-(0)=1,
\]
\[
 U=(1+t)L_+(t)+\e^{-t}P(t),\quad
 u=\frac{\e^{-j'(t)}}{1+t},\quad
 V=wrL_-(tr)+\frac{uL_-(tu)}{1-tr},
\]
\[
 J(s)=L_+(s)+s\e^{-s}(P(s)-P'(s)).
\]
Then \(f(t)/t=-\log t+U\) and \(\log X=-tV\).
On the right supporting branch, write \(X=y_f(t\sigma)\), \(Y=x_f(t\sigma)\).
Here \(\log Y=\log t+\log\sigma-\log(1+t\sigma)-j'(t\sigma)\) and \(\log\mu=\log t+\log r\).
The terms in \(\log t\) cancel in the normalized slack, giving
\begin{align*}
 \sigma J(t\sigma)&=V,\\
 \Delta_f(t)/t&=(w+1)U-V+\log\sigma-\log(1+t\sigma)
                -j'(t\sigma)+w\log r.
\end{align*}
As \(t\downarrow0\), \(X\to1\) implies \(t\sigma\to0\), since \(y_f\) is strictly decreasing with limit 1 only at zero. Thus \(\sigma J(t\sigma)=V\) and \(J(0)=1\) give the continuous extension \(\sigma(0)=V(0)=wr+\e^{-P(0)}>0\). On each cell near zero, the checker verifies the branch, a positive bracket for \(\sigma\), and \(t\sigma\le1\). The root is unique: \(\sigma\mapsto\sigma J(t\sigma)=-\log y_f(t\sigma)/t\) is strictly increasing for \(t>0\), and is the identity at zero.

The verifier evaluates complete finite polynomial expansions. Source inverse brackets are accepted only after outward-rounded endpoint inequalities prove them. When \(t\) and \(t\sigma(t)\) lie in the interiors of source pieces, possibly different ones,
\[
 \sigma'=\frac{V'-\sigma^2J'(t\sigma)}{J(t\sigma)+t\sigma J'(t\sigma)},
\]
with a checked positive denominator. The integral representations
\[
 L_+(u)=\int_0^1\frac{da}{1+ua},\qquad
 L_-(u)=\int_0^1\frac{da}{1-ua}
\]
give endpoint enclosures for these functions and their derivatives, including at zero. Away from zero, the function $u\mapsto\log Y_f(\e^u)$, with $u<0$, has derivative \(-s\), \(-1\), or \(-1/s\) on the three branches of \Cref{eq:Y}, where $s$ is the supporting parameter. These derivatives agree at the branch junctions.

Together with \Cref{eq:logseries,eq:expseries}, these identities enclose each slack and its derivative. An enclosure at the cell center plus the derivative interval times the displacement encloses the slack throughout the cell. When a cell crosses a source partition point, enclosures from every adjoining piece are combined. The continuous compositions are locally Lipschitz, and the stated derivative bounds hold almost everywhere; integrating those bounds justifies the same enclosure across a junction. The complete cover verifies the hypotheses of \Cref{lem:source}, which gives \Cref{eq:source-rate} with an error uniform in $b$.

\end{document}